\documentclass[11pt]{scrartcl}
\usepackage[utf8]{inputenc}
\usepackage[T1]{fontenc}
\usepackage[UKenglish]{babel}
\usepackage[para]{footmisc}

\usepackage{authblk}

\usepackage{amssymb}
\usepackage{amsmath}
\usepackage{amsthm}
\usepackage{mathtools}
\usepackage{enumitem}
\usepackage{hyperref}

\newcommand{\R}{\mathbb{R}}

\DeclareMathOperator{\im}{im}
\DeclareMathOperator{\spa}{span}
\DeclareMathOperator{\for}{\:for\:}

\DeclareMathOperator*{\argmax}{arg\,max}

\newcommand{\seq}[1]{(#1)_{n=1}^\infty}
\newcommand{\projseq}{\seq{(P_1 P_2)^n}}
\newcommand\normed[1]{\frac{#1}{\|#1\|}}

\newtheorem{lemma}{Lemma}

\newtheorem{proposition}{Proposition}
\theoremstyle{remark}
\newtheorem{remark}{Remark}
\newtheorem{example}{Example}

\begin{document}
\title{On angles, projections and iterations}
\author[1]{Christian Bargetz\footnote{christian.bargetz@uibk.ac.at}}
\author[2]{Jona Klemenc\footnote{jona.klemenc@uni-bonn.de}}
\author[3]{Simeon Reich\footnote{sreich@technion.ac.il}}
\author[ ]{Natalia~Skorokhod}
\affil[1]{University of Innsbruck}
\affil[2]{Rheinische Friedrich-Wilhelms-Universität Bonn}
\affil[3]{The Technion---Israel Institute of Technology}
\date{19th May 2020}

\DeclarePairedDelimiter\abs{\lvert}{\rvert}

\maketitle
\begin{abstract}
  We investigate connections between the geometry of linear subspaces and the convergence of the alternating projection method for linear projections. The aim of this article is twofold: in the first part, we show that even in Euclidean spaces the convergence of the alternating method is not determined by the principal angles between the subspaces involved. In the second part, we investigate the properties of the Oppenheim angle between two linear projections. We discuss, in particular, the question of existence and uniqueness of ``consistency projections'' in this context.
  \vskip2mm
  \noindent\textbf{Mathematics Subject Classification (2020).} 47J25, 46C05
  \vskip2mm
  \noindent\textbf{Keywords.} Alternating projection method, linear projections, principal angles, Oppenheim angle  
\end{abstract}

\section{Introduction}

The interest in the convergence of sequences of iterates of projections of various types goes back at least 
to the mid-twentieth century. J.~von~Neumann's article~\cite{Neumann1949} from 1949 can be considered one 
of the starting points of these investigations. In this article he shows that given a Hilbert space~$H$ 
and two closed subspaces $M, N\subset H$, with corresponding orthogonal projections $P_M$ and $P_N$, 
respectively, the sequence defined by 
\[
  x_0 \in H \qquad x_{2n+1} = P_Mx_{2n}\quad\text{and}\quad x_{2n+2} = P_Nx_{2n+1}
\]
converges in norm to $P_{M\cap N}x_0$ for every initial point $x_0 \in H$. An elementary geometric proof of 
von Neumann's theorem can be found in \cite{KR2004} This result was later generalised to 
the case of more than two subspaces by I.~Halperin in~\cite{Halperin}. In~\cite{Kayalar1988} 
S.~Kayalar and H.~L.~Weinert showed that the speed of convergence is determined by the Friedrichs numbers 
between the subspaces involved. This can be considered a geometric condition controlling the convergence 
behaviour.

Note that in all these cases the order in which the projections are iterated is of crucial importance. 
In~\cite{AmemiyaAndo}, I.~Amemiya and I.~Ando asked the question of whether convergence in norm can always be 
achieved provided that each projection appears infinitely often. This question was finally answered negatively 
by E.~Kopeck\'{a} and A.~Paszkiewicz in~\cite{KP2017:StrangeProducts}, where they give an example of three 
subspaces and an iteration order without convergence in norm. More information concerning this phenomenon can be 
found in \cite{K2019} and \cite{KM2014}. 

In Banach spaces, there are at least two natural generalisations of orthogonal projections: metric projections 
and linear projections. For the first one, the image is the point inside the subspace minimising the distance 
to the argument. It turns out that for iterations of metric projections, one cannot expect convergence 
of the iterates to the metric projection onto the intersection; see, for example,~\cite{Stiles}.

Recall that a linear mapping $P$ on a Banach space $X$ is called a \emph{linear projection} if it satisfies the condition 
that $P^2=P$. In this case, there are many positive results under the additional assumption that the projections
are of norm one. For example, if the Banach space~$X$ is uniformly convex, convergence of iterates of norm-one 
projections was established by R.~E.~Bruck and S.~Reich in~\cite{Reich}. This result was later generalised, 
for example, to further classes of Banach spaces by C.~Badea and Y.~I.~Lyubich in~\cite{BL2010:Geometric}. 
A dichtotomy for the speed of convergence of iterations of projections in Banach spaces which are 
uniformly convex of some power type has been exhibited by C.~Badea and D.~Seifert in~\cite{BS2017:Quantified}. 
More results on the convergence of the alternating algorithm for norm-one projections can be found 
in~\cite{DJRSSS2017:Nonoptimality}.

In the context of property~(T) for certain groups, I.~Oppenheim introduced in~\cite{Oppenheim2017a} 
an angle between linear projections in Banach spaces. This concept was developed further 
in~\cite{Oppenheim2017}, where a number of sufficient conditions for convergence of iterates of projections 
are given.

Iterations of non-orthogonal projections in Hilbert spaces, which then necessarily have a norm which is larger 
than one, are of interest in the context of discrete linear inclusions and of Skorokhod problems; 
see, for example,~\cite{KV2001:Polyhedral, HYPERPLANES}.

Since in most of the above results the convergence behaviour of the iterates of projections in determined 
or at least influenced by some kind of angle, one might hope that for non-orthogonal projections in Hilbert 
spaces the situation could be similar. More precisely, in the case of two linear projections, 
these projections are determined by two subspaces each---the range and the kernel. Moreover, in the case of 
Euclidean spaces, the concept of principal angles allows to determine the relative position of two subspaces 
up to an isometry. Therefore one could hope that these data might determine the convergence of the iterates 
of these projections.

The aim of this article is twofold: in the first part, we show that even in Euclidean spaces the convergence 
of the iterates is not determined by the principal angles between the subspaces involved. In the second part, 
we investigate the properties of the Oppenheim angle between two linear projections and provide an example 
which shows that the modification of the definition of this angle introduced in~\cite{Oppenheim2017} is indeed 
necessary.

\section{Preliminaries and Notation}

\subsection{Principal Angles}

Principal angles are used to describe the geometric configuration of two subspaces of a real Hilbert space $H$ up to orthogonal mappings. Given two finite-dimensional subspaces $S_1, S_2\subseteq H$ and denoting by~$q$ the minimum 
of the dimensions of $S_1$ and $S_2$, the principal angles
\begin{equation*}
  \theta _1, ..., \theta _q \in \left[0, \frac{\pi}{2}\right]
\end{equation*}
and the corresponding principal vectors
\[
    u_1, ..., u_q \in S_1 \qquad  v_1, ..., v_q \in S_2
\]
are defined inductively by
\begin{equation}\label{eq:PrincipalAngles}
  \begin{split}
    u_k, v_k &= \argmax \{|\langle u_k, v_k \rangle| : \langle u_k, u_i \rangle = \langle v_k, v_i \rangle = \delta_{k, i} \for i = 1, ... k\} \\
    \theta_k &= \arccos(|\langle u_k, v_k \rangle|)
  \end{split}
\end{equation}
for $k=1,\ldots, q$. 
The principal angles can also be represented in terms of the orthogonal projections $P_{S_1}$ and $P_{S_2}$ onto $S_1$ and $S_2$, respectively. 
More precisely, 
\begin{equation*}
  \begin{split}
    \theta_k = \arccos(\sqrt{\lambda_k}),
  \end{split}
\end{equation*}
where \(\lambda_1 \geq ... \geq \lambda_q\) are the first \(q\) eigenvalues of the restriction of $P_{S_2} P_{S_1}$ to $S_2$. 
This formula allows for a direct computation of the principal angles, thus avoiding the optimisation problems in~\eqref{eq:PrincipalAngles}. Moreover, it has the advantage that it also makes sense for 
infinite-dimensional subspaces; see, for example,~\cite{AnglesInfinite}. \\
The principal angles between two subspaces define them completely up to a simultaneous rotation. 
This means that every rotation-invariant function of two subspaces can be written as a function 
of the principal angles between them; 
for example, the Dixmier angle and the Friedrichs angle are the smallest and the smallest non-zero principal angle, 
respectively.
\\
Since the function which maps subspaces to their orthogonal complements commutes with rotations, knowing all the angles between two subspaces \(S_1\) and \(S_2\) is equivalent to knowing all the angles 
between \(S_1\) and \(S_2^\bot\). \\
We denote by \(\Theta(S_1, S_2)\) the ordered tuple of the principal angles between two subspaces 
\(S_1\) and \(S_2\). \\
A more detailed exposition of these angles, where the relations between various approaches to angles between 
subspaces---including principal angles and directed distances---are examined, can be found 
in~\cite[Chapter 5.15]{matrix2000meyer} and in the very nice survey article~\cite{BS}.

\subsection{The Cross-Ratio of projective points}
For four distinct points $a_1, a_2, a_3, a_4$ of the projective line $\mathbb{P}^1(\mathbb{R})$, the \textit{cross-ratio} of these points, denoted by \([a_1, a_2, a_3, a_4] \in \R \cup \{\infty\}\), is defined by
\begin{equation*}
  [a_1, a_2, a_3, a_4] = \frac{
    \det
    \begin{pmatrix}
      \lambda_3 & \lambda_1 \\ \mu_3 & \mu_1
    \end{pmatrix}
    \det
    \begin{pmatrix}
      \lambda_4 & \lambda_2 \\ \mu_4 & \mu_2
    \end{pmatrix}
  }{
    \det
    \begin{pmatrix}
      \lambda_3 & \lambda_2 \\ \mu_3 & \mu_2
    \end{pmatrix}
    \det
    \begin{pmatrix}
      \lambda_4 & \lambda_1 \\ \mu_4 & \mu_1
    \end{pmatrix}
  }
\end{equation*}
where \(\frac{y}{0} = \infty\) and \(a_i = p(\lambda_i, \mu_i)\), that is, \(\lambda_i, \mu_i\) are the 
homogeneous coordinates of \(a_i\). In our application, the denominator will never be zero, 
and so we can always assume that \([a_1, a_2, a_3, a_4] \in \R\). We will need the following formula 
for the cross-ratio:
\begin{lemma}\label{lemma:cross-ratio-formula}
  Let \(a \in (\R^2)^4\). Then
  \begin{equation*}
    \frac{\langle a_1, a_3 \rangle \langle a_2, a_4 \rangle}{\langle a_1, a_4 \rangle \langle a_2, a_3 \rangle} = [\spa(a_1), \spa(a_2), \spa(a_3)^\bot, \spa(a_4)^\bot]
  \end{equation*}
\end{lemma}
\begin{proof}
  Using \((x_2, -x_1)\) as homogeneous coordinates for \(\spa(x)^\bot\) and the Leibniz formula for the 
determinant, we can directly obtain this assertion from the definition.
\end{proof}
The behaviour of the cross-ratio function with respect to permutations is well known. Since we will use the 
following lemma later, we state it at this point without proof (which can be found, for example, in \cite[pp. 
123--126]{Geometry1}).
\begin{lemma}\label{lemma:cross-ratio-argument-switch}
  The cross-ratio satisfies
  \begin{equation*}
      [a, b, c, d] = [d, c, b, a] \qquad\text{and}\qquad [a, c, b, d] = 1 - [a, b, c, d],
  \end{equation*}
  where $a, b, c, d$ are pairwise distinct one-dimensional subspaces of \(\R^2\).
\end{lemma}

\subsection{The Oppenheim angle between linear projections}

Let $P_1,P_2$ be two bounded linear projections in a Banach space. Assume that there is a bounded linear 
projection $P_{12}$ onto the intersection of the images of $P_1$ and $P_2$ satisfying 
\begin{equation}\label{eq:ConsistProj}
  P_{12}P_1=P_{12}\qquad\text{and}\qquad P_{12}P_2=P_{12}.
\end{equation}
Note that $P_{12}$ is required to be surjective. Using these data, we define
\[
  \cos_{P_{12}}(\angle (P_1,P_2)) = \max\{\|P_1(P_2-P_{12})\|,\|P_2(P_1-P_{12})\|\}
\]
and
\[
  \cos(\angle (P_1,P_2)) = \inf\{\cos_{P_{12}}(P_1,P_2)\colon P_{12}\;\text{projection onto}\;\im P_1\cap\im P_2\;\text{satisfying~\eqref{eq:ConsistProj}}\}.
\]
In the case of two orthogonal projections $P_1, P_2$ in a Hilbert space, the above angle coincides 
with the Friedrichs angle between the images of $P_1$ and $P_2$. The subtraction of the projection 
$P_{12}$ in the definition above plays the role of the quotient in the definition of the Friedrichs angle. 
There need not always be such a projection $P_{12}$. Moreover, the intersection of the images 
of $P_1$ and $P_2$ need not even be complemented. Two projections $P_1$ and $P_2$ are called
\emph{consistent} if such a projection does exist. We also call a projection $P_{12}$ with the above 
properties a \emph{consistency projection}. Our main interest in this angle lies in the fact that 
a large Oppenheim angle, that is, a small cosine in the above definition, implies that the iterations 
$(P_1P_2)^{n}$ converge uniformly to a (consistency) projection onto $\im P_1\cap \im P_2$. 
For a detailed discussion of these angles, we refer the reader to~\cite{Oppenheim2017}.

\section{Classification of the Convergence Behaviour Through Principal Angles}

In all what follows, let $X$ be a Euclidean space. All geometric characterisations of the convergence 
behaviour we provide in this article are based on the two-dimensional case. Given two subspaces $M,N\subset\mathbb{R}^{n}$ we use the notation $M\oplus N$ for the direct sum of $M$ and $N$, that is, this notation indicates that $M\cap N=\{0\}$.

\subsection{The Two-Dimensional Case}
For two two-dimensional projections \(P_1, P_2: \R^2\rightarrow \R^2\), the convergence behaviour of \(\projseq\) is determined by the geometric relation between the nullspaces and ranges of \(P_1, P_2\).
The question of the convergence behaviour of the sequence of iterates $((P_1P_2)^{n})_{n=1}^{\infty}$ is trivial if either of the projections is the identity or the zero mapping. Therefore we restrict ourselves to the case where the range of both projections is a one-dimensional subspace of $\mathbb{R}^{2}$.

\begin{proposition}\label{prop:2d}
  Let  \(P_1, P_2: \mathbb{R}^{2} \rightarrow \mathbb{R}^{2}\) be two non-trivial projections, 
that is, neither of them is the identity or zero. We set $R_1 = \im P_1$, $N_1 = \ker P_1$, $R_2= \im 
P_2$ and 
$N_2 = \ker P_2$. The composition $P_1P_2$ has at most one non-zero eigenvalue \(\lambda\) and it satisfies
  \begin{equation*}
    \lambda = [R_1, R_2, N_2, N_1].
  \end{equation*}
  In particular, for the spectral radius $\rho(P_1 P_2)$ we have
  \begin{equation*}
    \rho(P_1 P_2) = \frac{s(R_1, N_2) s(R_2, N_1)}{s(R_1, N_1) s(R_2, N_2)},
  \end{equation*}
  where $s(M,N):=\sqrt{1-c(M,N)^2}$ for subspaces $M,N\subset\mathbb{R}^{2}$ and $c(M,N)$ is the Friedrichs angle of $M$ and $N$. The iterates $(P_1P_2)^{n}$ converge to zero if and only if the geometric condition
  \[
    \frac{s(R_1, N_2) s(R_2, N_1)}{s(R_1, N_1) s(R_2, N_2)} < 1
  \]
  is satisfied.
\end{proposition}

\begin{proof}
  Let \(w_1, v_1, w_2, v_2 \in \mathbb{R}^2\) be such that \(P_1 = w_1 v_1^\ast\), \(P_2 = w_2 v_2^\ast\).
  For \(k= 1, 2\), set \(R_k = \im P_k\), \(N_k = \ker P_k\) and observe that \(\spa(v_k) = N_k^\bot\), \(\spa(w_k) = R_k \neq \{0\}\) and, since \(P_k\) is a projection,
  \begin{equation*}
    w_k = P_k(w_k) = w_k \langle v_k, w_k \rangle,
  \end{equation*}
  and thus \(1 = \langle v_k, w_k \rangle\).
  
  Since \(R_1 = \spa(w_1)\), any non-zero eigenvector of \(P_1 P_2\) must be a multiple of \(w_1\). Calculating $P_1P_2w_1$, we get
  \begin{equation*}
    \begin{split}
      P_1 P_2 w_1 &= w_1 v_1^\ast(w_2) v_2^\ast(w_1) = \frac{\langle v_1, w_2 \rangle \langle v_2, w_1 \rangle}{\langle v_1, w_1 \rangle \langle v_2, w_2 \rangle} w_1
    \end{split}
  \end{equation*}
  and hence, using Lemma~\ref{lemma:cross-ratio-formula}, we obtain
  \begin{equation*}
    \begin{split}
      \lambda &= \frac{\langle v_1, w_2 \rangle \langle v_2, w_1 \rangle}{\langle v_1, w_1 \rangle \langle v_2, w_2 \rangle} =
      \frac{\langle w_1, v_2 \rangle \langle w_2, v_1 \rangle}{\langle w_1, v_1 \rangle \langle w_2, v_2 \rangle} \\
      &= [\spa(w_1), \spa(w_2), \spa(v_2)^\bot, \spa(v_1)^\bot] = [R_1, R_2, N_2, N_1]
    \end{split}
  \end{equation*}
  and
  \begin{equation*}
    \rho(P_1 P_2) = \frac{| \langle w_1, v_2 \rangle | | \langle w_2, v_1 \rangle | }{| \langle w_1,v_1 \rangle | | \langle w_2, v_2 \rangle |} = 
    \frac{| \langle \normed{w_1}, \normed{v_2} \rangle | | \langle \normed{w_2}, \normed{v_1} \rangle | }{| \langle \normed{w_1}, \normed{v_1} \rangle | | \langle \normed{w_2}, \normed{v_2} \rangle |} 
    = \frac{s(R_1, N_2) s(R_2, N_1)}{s(R_1, N_1) s(R_2, N_2)}\text{.}
  \end{equation*}
  In particular, we see that convergence occurs if and only if the modulus of the above number is smaller than one.
\end{proof}

\begin{remark}
  Note that for two projections $P_1$ and $P_2$ in $\mathbb{R}^{2}$ with distinct one-dimensional images, 
the zero mapping is the uniquely determined projection onto the intersection of these ranges. 
Using the above result, we see that
  \[
    \frac{s(R_1, N_2) s(R_2, N_1)}{s(R_1, N_1) s(R_2, N_2)} = \rho(P_1P_2) \leq \|P_1P_2\| 
= \cos \angle_{0}(P_1,P_2) = \cos \angle(P_1,P_2),
  \]
  that is, the iterates converge whenever the ``cosine'' of the Oppenheim angle between the 
projections is smaller than one. Moreover, the above discussion shows that, in this particular case, 
there is a relation between the Oppenheim angle and the principal angles.
\end{remark}

\subsection{The three-dimensional case: Some additional information might be needed}

Also in the three-dimensional case, we restrict ourselves to non-trivial projections, 
that is, we exclude both the identity mapping and the zero mapping. In this section let 
$P_1,P_2$ 
be two non-trivial projections in $\mathbb{R}^{3}$. In order to simplify the notation, we 
use the abbreviations
\[
  R_1 = \im P_1, \qquad N_1 = \ker P_1, \qquad R_2 = \im P_2 \qquad\text{and}\qquad N_2 = \ker P_2.
\]
We start our investigation of the convergence behaviour with the observation that the problem is at its core still a two-dimensional problem. We call an eigenvector of the composition  $P_1P_2$ non-trivial if the corresponding eigenvalue is neither zero nor one.

\begin{lemma}\label{lemma:eigenvector_geometric}
  For every non-trivial eigenvector \(v\) of \(P_1 P_2\), there exists a two-dimensional subspace \(E_v\) such that the intersections of \(E_v\) with \(R_1, R_2, N_1, N_2\) are all one-dimensional with trivial pairwise intersection and \(v \in E_v\).
\end{lemma}
\begin{proof}
  Let \(\lambda\) be the eigenvalue corresponding to \(v\) (that is \(P_1 P_2 v = \lambda v\)). Set \(w = P_2(v)\), \(E_v = \spa(w, v)\) (\(\lambda \neq 0\) implies \(w \neq 0\)). Then
  \begin{equation*}
    \begin{split}
      \lambda v \in R_1 \cap E_v, \qquad   \lambda v - w = P_1(w) - w \in N_1 \cap E_v,\\
       w \in R_2 \cap E_v\quad\text{and}\quad  w - v = P_2(v) - v \in N_2 \cap E_v.
    \end{split}
  \end{equation*}
  Note that \(w \in R_1\) or \(v \in R_2\) would mean that \(w = v = \lambda v\), 
while \(w \in N_1\) or \(v \in N_2\) would mean that \(P_1 P_2 v = 0\), both contrary to the assumption 
that \(v\) is non-trivial. Therefore \(v\) and \(w\) are linearly independent, 
all the intersections are one-dimensional and \(E_v\) is two-dimensional. 
Also using this, we can easily reproduce the triviality of the pairwise intersections. 
\end{proof}

The above lemma implies, in particular, that we still can interpret the ranges and kernels of the 
projections under consideration as projective points. Moreover, the cross-ratio of these points 
still carries the vital information regarding the convergence.

\begin{lemma}\label{lemma:nontrivial_ev_formula}
  Let \(v\) be an eigenvector corresponding to the non-trivial eigenvalue \(\lambda\) of the operator \(P_1 P_2\) and \(E_v\) be the associated subspace established in Lemma~\ref{lemma:eigenvector_geometric}. Set
  \begin{equation*}
    R_1' = R_1 \cap E_v, \qquad R_2' = R_2 \cap E_v, \qquad N_1' = N_1 \cap E_v \qquad\text{and}\qquad N_2' = N_2 \cap E_v
  \end{equation*}
  Then 
  \begin{equation*}
    \lambda = [R_1', R_2', N_2', N_1'],
  \end{equation*}
  where the cross-ratio is meant to be taken on \(E_v\).
\end{lemma}
\begin{proof} 
  First, we show that \(P_k(E_v) \subseteq E_v\) for \(k \in \{1,2\}\). To this aim, first observe that Lemma~\ref{lemma:eigenvector_geometric} implies that \({E_v = R_k' \oplus N_k'}\). 
  Given \(x \in E_v\), we choose \(r \in R_k', n \in N_k'\) such that \(x = r + n\). Then
  \begin{equation*}
    P_k(x) = P_k(r) + P_k(n) = r \in E_v.
  \end{equation*} 
  Hence the mappings \(P_k': E_v \rightarrow E_v, x \mapsto P_k(x)\) are well-defined projections for \(k \in \{1, 2\}\). 
  Since, by construction, \(\lambda\) is the non-trivial eigenvalue of \(P_1' P_2'\), 
  we can use Proposition~\ref{prop:2d} to complete the proof:
  \begin{equation*}
    \lambda  = [\im P_1', \im P_2', \ker P_2', \ker P_1'] = [R_1', R_2', N_2', N_1'].
  \end{equation*}
\end{proof}
A plane with the properties of \(E_v\) in Lemma~\ref{lemma:eigenvector_geometric} is conversely always associated 
to a non-trivial eigenvector:
\begin{lemma}
  Let \(P_1, P_2\) be projections and let \(E\) be a two-dimensional subspace of \(\mathbb{R}^3\) 
such that the intersections of \(E\) with \(R_1, R_2, N_1, N_2\) are all one-dimensional with trivial pairwise intersection. 
Then \(E \cap R_1\) is an eigenspace of \(P_1 P_2\) corresponding to a non-trivial eigenvalue. 
\end{lemma}

\begin{proof}
  As in the proof of Lemma~\ref{lemma:nontrivial_ev_formula} above, the well-definedness 
of the two projections \({P_k': E \rightarrow E, x \mapsto P_k(x)}\) follows from 
\(E = (R_k \cap E) \oplus (N_k \cap E)\) for \(k \in \{1, 2\}\).
  Since \(R_1' = R(P_1') \) is one-dimensional, it is an eigenspace of \(P_1' P_2'\) corresponding to the eigenvalue 
  \begin{equation*}
    \lambda = [\im P_1', \im P_2', \ker P_2', \ker P_1'] = [R_1 \cap E, R_2 \cap E, N_2 \cap E, N_1 \cap E].
  \end{equation*}
  By construction, \(R_1'\) is also an eigenspace of \(P_1 P_2\) corresponding to \(\lambda\), and since the arguments in the 
cross-ratio function are pairwise distinct, \(\lambda \notin \{0, 1\}\) (see, for example, \cite[Proposition 
6.1.3]{Geometry1}). \end{proof}

In order to formulate a characterisation of convergence in three dimensions, we need a geometric lemma regarding the connection of angles and directed distances for subspaces. Recall that for subspaces $M,N\subseteq X$ the directed distance $\delta(M,N)$ is defined by
\[
  \delta(M,N) = \sup \{d(x,N)\colon x\in M, \|x\|=1\}.
\]
A simple computation shows that $\delta(M,N)=\sup\{\|P_{N^\perp}x\|\colon x\in M, \|x\|=1\}$.

\begin{lemma}\label{lemma:gap_ratio_constant}
  Let \(H\) be a real Hilbert space, let \(S_1, S_2, V\) be three one-dimensional, pairwise distinct subspaces of \(H\) such that \(V \subseteq S_1 \oplus S_2\), and let \(W\) be another subspace of \(H\) such that \((S_1 \oplus S_2) \cap W = \{0\}\). Then
  \begin{equation*}
    \frac{s(S_1, V)}{s(S_2, V)} = \frac{\delta(S_1, V \oplus W)}{\delta(S_2, V \oplus W)}.
  \end{equation*}
\end{lemma}
\begin{proof}
  Let \(k \in \{1,2\}\) and let \(s_k \in S_k\) satisfy \(\|s_k\| = 1\). 
  Note that
  \[
    s(S_k, V) = \|P_{V^\bot}(s_k)\|\quad\text{and}\quad
    \delta(S_k, V \oplus W) = \|P_{(V \oplus W)^\bot}(s_k)\|
  \]
  because both $S_k$ and $V$ are one-dimensional. From
  \[
    N(P_{V^\bot}) = V \subseteq V \oplus W = N(P_{(V \oplus W)^\bot})
  \]
  we may conclude that \(P_{(V \oplus W)^\bot} = P_{(V \oplus W)^\bot} P_{V^\bot}\).
  Since the subspaces \(S_1, S_2, V\) are pairwise distinct and \(S_1 \subseteq S_2 \oplus V\), we may pick  \(c \in \R\) and \(v \in V\) such that \(s_1 = c s_2 + v\). Then,
  \[
    P_{V^\bot}(s_1) = P_{V^\bot}(c s_2) + P_{V^\bot}(v) = c \: P_{V^\bot}(s_2).
  \] 
  Comparing the norms of these expressions, we see that \(P_{V^\bot}(s_1) = c \: P_{V^\bot}(s_2)\) with a number $c$ satisfying \(|c| = \frac{s(S_1, V)}{s(S_2, V)}\).
  Therefore, we may conclude that
  \begin{equation*}
    \begin{split}
      \delta(S_1, V \oplus W) &= \|P_{(V \oplus W)^\bot}(s_1)\| = \|P_{(V \oplus W)^\bot}P_{V^\bot}(s_1)\| \\
      &= \|c \: P_{(V \oplus W)^\bot} P_{V^\bot}(s_2) \|= |c| \|\: P_{(V \oplus W)^\bot}(s_2) \|\\
      &=  \frac{s(S_1, V)}{s(S_2, V)} \: \delta(S_2, V \oplus W).
    \end{split}
  \end{equation*}
  Finally, as \((S_1 \oplus S_2) \cap W = \{0\}\) and hence \(S_2 \nsubseteq V \oplus W\), we have  \(\delta(S_2, V \oplus W) \neq 0\) which finishes the proof.
\end{proof}

\begin{proposition}\label{prop:3d_formula}
  Let \(P_1, P_2: \R^3 \rightarrow \R^3\) be two projections with two-dimensional images. There is at most one non-trivial eigenvalue and if it exists it satisfies the equation
  \begin{equation*}
    |\lambda| = \frac{\delta(N_2, R_1) \delta(N_1, R_2)}{\delta(N_1, R_1) \delta(N_2, R_2)}.
  \end{equation*}
  In particular, the iterates $(P_1P_2)^{n}$ converge if and only if 
  \[
    \frac{\delta(N_2, R_1) \delta(N_1, R_2)}{\delta(N_1, R_1) \delta(N_2, R_2)} < 1,
  \]
  that is, the convergence is determined by the angles between the ranges and kernels.
\end{proposition}

\begin{proof}    
  Assuming the existence of at least one non-trivial eigenvector $v$ with corresponding eigenvalue \(\lambda\), we see that $N_1\cap N_2=\{0\}$, the only possible choice of \(E_v\) is \(N_1 \oplus N_2\), and \(R_1 \cap R_2 \cap E_v = \{0\}\).
  We set
  \begin{equation*}
      R_1' = R_1 \cap E_v, \qquad R_2' = R_2 \cap E_v \qquad\text{and}\qquad Z = R_1 \cap R_2.
  \end{equation*}
  Then, by comparing dimensions, we observe that \(R_1 = R_1' \oplus Z\) and \(R_2 = R_2' \oplus Z\). 
  Using Lemmata~\ref{lemma:nontrivial_ev_formula} and~\ref{lemma:gap_ratio_constant}, we conclude that
  \begin{align*}
    |\lambda| &= |[R_1', R_2', N_2, N_1]|  = \frac{s(N_2, R_1')}{s(N_1, R_1')} \: \frac{s(N_1, R_2')}{s(N_2, R_2')} = \frac{\delta(N_2, R_1' \oplus Z)}{\delta(N_1, R_1' \oplus Z)} \: \frac{\delta(N_1, R_2' \oplus Z)}{\delta(N_2, R_2' \oplus Z)}\\
              &= \frac{\delta(N_2, R_1) \delta(N_1, R_2)}{\delta(N_1, R_1) \delta(N_2, R_2)},
  \end{align*}
  as claimed.
\end{proof}

\begin{remark}\label{rem:3dimRem}
  We conclude this section with a few observations concerning the validity of the above characterisation of convergence 
  of the iterates.
  \begin{enumerate}[wide, labelwidth=!, labelindent=0pt]
  \item Using the behaviour of the cross-ratio with respect to permutations of its arguments 
    stated in Lemma~\ref{lemma:cross-ratio-argument-switch}, we obtain that for the case of 
    two linear projections $P_1,P_2$ on $\mathbb{R}^{3}$ with one-dimensional ranges, we also see that there is at most one 
    non-trivial eigenvalue and if it exists, then it satisfies the equation
    \begin{equation*}
      |\lambda| = \frac{\delta(R_1, N_2) \delta(R_2,N_1)}{\delta(R_1,N_1) \delta(R_2,N_2)}.
    \end{equation*}
    In particular, the iterates $(P_1P_2)^{n}$ converge if and only if 
    \[
      \frac{\delta(R_1, N_2) \delta(R_2,N_1)}{\delta(R_1,N_1) \delta(R_2,N_2)} < 1.
    \]
    This shows that also in this case, convergence of the iterates is determined by the angles between 
    the ranges and the kernels.
  \item Copying the above arguments, it is possible to show the same characterisation for projections 
    on Hilbert spaces, where both have either one-dimensional images or one-dimensional kernels.
  \item The ``mixed case'', that is, the case where one projection has a one-dimensional range and the 
    other projection has a one-dimensional kernel, is more complicated. Let $P_1$ and $P_2$ be two projections on 
    $\mathbb{R}^{3}$, and assume that $P_1$ has a one-dimensional range and $P_2$ has a two-dimensional one. 
    Using the results of this section, we may conclude that there is a unique non-trivial eigenvalue $\mu$ 
    of
    \[
      P(\im P_1,\ker P_1)P(\ker P_2,\im P_2)
    \]
    and that the unique non-trivial eigenvalue of $P_1P_2$ is $\lambda = 1-\mu$. 
    So, if
    \begin{equation}\label{eq:cond}
      \frac{\delta(\im P_1, \ker P_2) \delta(\im P_2, \ker P_1)}{\delta(\im P_1,\ker P_1) \delta(\im P_2, \ker P_2)} < 1.
    \end{equation}
    and the cross-ratio $[\im P_1, \ker P_2, \im P_2, \ker P_1]>0$, we have convergence of the iterates $(P_1P_2)^{n}$. 
    In order to show that condition~\eqref{eq:cond} is not enough, we consider the following example.
    Take the vectors
    \begin{equation*}
      w_1 = \begin{pmatrix} 1 \\ -1 \\ 0 \end{pmatrix}, \quad
      w_2 = \begin{pmatrix} 1 \\ 1 \\ 0 \end{pmatrix}, \quad
      w_3 = \begin{pmatrix} 1 \\ 3 \\ 1 \end{pmatrix}, \quad
      w_4 = \begin{pmatrix} 0 \\ 1 \\ -1 \end{pmatrix} \quad\text{and}\quad
      w_4' = \begin{pmatrix} 1 \\ 0 \\ 1 \end{pmatrix}
    \end{equation*}
    and set \(S_i = \spa(w_i)\) for \(i = 1, 2\); \(S_i = \spa(w_i)^\bot\) 
    for \(i=3, 4\); \(S_4' = \spa(s_4')^\bot\). A direct computation shows that
    \[
      \Theta(S_4, S_i) = \Theta(S_4', S_i)\qquad\text{for}\qquad i=1,2,3.
    \]
    On the other hand, setting \(V_i = S_i \cap (S_1 \oplus S_2)\), \(V_4' = S_4' \cap (S_1 \oplus S_2)\), we get
    \begin{equation*}
      \begin{split}
        \frac12 = [V_1, V_2, V_3, V_4] = - [V_1, V_2, V_3, V_4'].
      \end{split}
    \end{equation*}
    Using the above arguments, we see that the sequence \(\seq{P(S_1, S_4)(P(S_3, S_2))^n}\) converges 
    whereas the sequence \(\seq{(P(S_1, S_4')P(S_3, S_2))^n}\) does not. Since all the principal angles are the same 
    in both cases, they do not determine the convergence behaviour on their own.
  \end{enumerate}
\end{remark}
\subsection{Higher dimensions: Angles are not enough}
For dimensions higher than three, a characterisation using only the principal angles cannot work. We show this 
by giving a counterexample. It is somewhat similar to the one given in Remark~\ref{rem:3dimRem}, but in contrast to 
the situation there, in higher dimensions it seems to be unclear how to overcome the problem. 
The counterexample is built by combining two two-dimensional examples in a specific way. 
In order to do this, we first need a simple observation on operator matrices.
For two operators \(T_1: H_1 \rightarrow H_1\), \(T_2: H_2 \rightarrow H_2\) on Hilbert spaces $H_1$ and $H_2$, 
we denote the operator matrix
\begin{equation*}
  \begin{split}
    \begin{bmatrix}
      T_1 & 0 \\
      0 & T_2
    \end{bmatrix}: 
    H_1 \oplus H_2 &\rightarrow H_1 \oplus H_2\\
    (h_1, h_2) &\mapsto (T_1(h_1), T_2(h_2)) 
  \end{split}
\end{equation*}
by \(T_1 \oplus T_2 \). We consider the case of four projections $P_1$ and $P_2$ on $H_1$, and $P_3$ and $P_4$ on $H_2$. 
Then $P_1\oplus P_3$ and $P_2\oplus P_4$ are projections on $H_1\oplus H_2$ satisfying
\begin{gather*}
  \im P_1\oplus P_3 =  \im P_1 \oplus \im P_3, \qquad \ker P_1\oplus P_3 =  \ker P_1 \oplus \ker P_3,\\
  \im P_2\oplus P_4 =  \im P_2 \oplus \im P_4 \quad\text{and}\quad \ker P_2\oplus P_4 =  \ker P_2 \oplus \ker P_4  
\end{gather*}
as can be seen by a direct computation. Moreover, the spectrum of $(P_1\oplus P_3)(P_2\oplus P_4)$ is just the union of 
the spectra of $P_1P_2$ and $P_2P_4$.

The final observation needed for the example is that principal angles between two direct sums of subspaces 
are nothing but the combined principal angles between the individual subspaces in both summands.

\begin{example}
  We construct a counterexample in \(\R^4\). More precisely, we construct four projections 
\(P_1, P_2, P_1', P_2'\) such that all the principal angles between the ranges and nullspaces of \(P_1, P_2\) 
on one hand, and \(P_1', P_2'\) on the other, are identical, but \(\rho(P_1 P_2) = 0\), 
while \(\rho(P_1' P_2') = 2\).

  First, we fix \(\varphi \in (0, \frac\pi2)\). For \(\theta \in (0, \frac\pi2)\), we set
  \[
    e(\theta) =
    \begin{pmatrix}
      \cos(\theta) \\
      \sin(\theta)
    \end{pmatrix}, \qquad S(\theta) = \spa(e(\theta))
    \qquad\text{and}\qquad
    \begin{pmatrix}
      x \\
      y
    \end{pmatrix}^\bot = 
    \begin{pmatrix}
      y \\
      -x
    \end{pmatrix}.
  \]
  Now let \(s = \pm 1\) and set
  \[
    P_1^1 = P(S(0), S(\frac\pi2 + \varphi))\qquad\text{and}\qquad P_{2, s}^1 = P(S(\frac\pi2 + s\varphi), S(\frac\pi2)).
  \]
  Since \(\im (P_{2, 1}^1) = \ker(P_1^1)\) we have \(P_1^1 P_{2, 1}^1 = 0\). 
Calculating the unique non-trivial eigenvalue of \(P_1^1 P_{2, -1}^1\), we get 
  \begin{align*}
    |\lambda| = \rho(P_1^1 P_{2, -1}^1) &= \abs{\frac{\langle e(0), e(\frac\pi2)^\bot \rangle \langle e(\frac\pi2 - \varphi), e(\frac\pi2 + \varphi)^\bot \rangle}{\langle e(0), e(\frac\pi2 + \varphi)^\bot \rangle \langle e(\frac\pi2 - \varphi), e(\frac\pi2)^\bot \rangle}} = \abs{\frac{\cos(\frac\pi2 - 2 \varphi)}{\cos(\varphi)\cos(\frac\pi2 - \varphi)}}\\
    & = \abs{\frac{\sin(2 \varphi)}{\cos(\varphi) \sin(\varphi)}}  
      = 2 \abs{\frac{\cos(\varphi)\sin(\varphi)}{\cos(\varphi)\sin(\varphi)}} = 2.
  \end{align*}  
  In order to construct the other pair of projections, we set
  \[
    P_1^2 = P(S(0), S(\frac\pi2 + \varphi))\qquad\text{and}\qquad P_{2, s}^2 = P(S(\frac\pi2 - s\varphi), S(0)).
  \]
  Since \(R(P_1^2) \subseteq N(P_{2, s}^2)\) we have \(\rho(P_1^2 P_{2, s}^2) = 0\).

  We combine these projections by setting
  \[
    P_1 := P_1^1\oplus P_1^2, \qquad\text{and}\qquad P_{2,s} := P_{2,s}^1\oplus P_{2,s}^2
  \]
  for $s=\pm 1$. Since the principal angles between direct sums of subspaces are just the combination 
of the principal angles between the individual spaces, the principal angles between the ranges and 
kernels of $P_1$ and $P_{2,-1}$ are the same as the ones between the ranges and kernels of $P_1$ and 
$P_{2,1}$. On the other hand, we have \(\rho(P_1 P_{2, 1}) = 0\) and \(\rho(P_1 P_{2, -1}) = 2\), 
that is, although the principal angles agree, nevertheless the convergence behaviour is vastly 
different.  
\end{example}
	
\section{Some remarks on angles between linear projections}

Let $P_1,P_2$ be two bounded linear projections in a Banach space. Recall that in order to define 
the Oppenheim angle between $P_1$ and $P_2$, we need a projection~$P_{12}$ which satisfies $P_{12}P_1 = P_{12}$ and 
$P_{12}P_2 = P_{12}$. As noted in Remark~2.6 in~\cite[p.~346]{Oppenheim2017} such a projection need not be unique. Since 
in~\cite{Oppenheim2017} no example is given, we now give a simple example illustrating this phenomenon.

\begin{example}
  We consider the projections
  \[
    P_1\colon\mathbb{R}^{3}\to\mathbb{R}^{3}, \qquad (x,y,z) \mapsto (x+y,0,z)
  \]
  and
  \[
    P_2\colon\mathbb{R}^{3}\to\mathbb{R}^{3}, \qquad (x,y,z) \mapsto (0,x+y,z).
  \]
  The intersection of the images of $P_1$ and of $P_2$ is the $z$-axis, the projections
  \[
    P_{12}\colon\mathbb{R}^{3}\to\mathbb{R}^{3}, \qquad (x,y,z) \mapsto (0,0,z)
  \]
  and
  \[
    P_{12}'\colon\mathbb{R}^{3}\to\mathbb{R}^{3}, \qquad (x,y,z) \mapsto (0,0,x+y+z)
  \]
  are both projections onto $\im P_1\cap \im P_2$. Observe that
  \[
    P_{12}P_1 = P_{12}, \qquad P_{12}P_{2} = P_{12}, \qquad P_{12}'P_1 = P_{12}' \qquad\text{and}\qquad P_{12}'P_{2} = P_{12}'.
  \]
  So both projections are admissible in the definition of the Oppenheim angle. A direct computation shows that, with $\|\cdot\|_1$ denoting the operator norm induced by the $\ell^1$ norm on $\mathbb{R}^3$,
  \[
    \|P_1(P_2-P_{12})\|_1 = \|P_2(P_1-P_{12})\|_1 =1
  \]
  but
  \[
    \|P_1(P_2-P_{12}')\|_1 = \|P_2(P_1-P_{12}')\|_1 =2
  \]
  This shows that these projections result in different values for the Oppenheim angle. 
For the Euclidean norm these two projections result in the same Oppenheim angle. Taking on the other hand
  \[
    P_{12}''\colon\mathbb{R}^{3}\to\mathbb{R}^{3}, \qquad (x,y,z) \mapsto (0,0, z+ 2x+2y)
  \]
  and $P_{12}'$, we obtain different angles for the Euclidean norm as well. Note that in the first case, 
we even have $\|P_{12}\|_{1}=\|P_{12}'\|_{1}=1$.
\end{example}

In infinite dimensional Banach spaces, even the question of whether two projections $P_1$ and $P_2$ 
are consistent, that is, if there is a projection $P_{12}$ onto the intersection of the ranges of 
$P_1$ and $P_2$ such that $P_{12}P_1 = P_{12}P_2 = P_{12}$, is of interest. 
Note that there are complemented subspaces with the property that their intersection is no longer 
complemented. In other words, it might happen that not only there is no projection satisfying 
the above condition, but that there is no bounded projection at all.

On the positive side, we can mention the following result of R.~E.~Bruck and S.~Reich:

\begin{proposition}[{Theorem~2.1 in~\cite[p.~464]{Reich}}]\label{prop:Reich}
  Let~$X$ be a uniformly convex space and let $P_{1},\ldots,P_{k}$ be linear norm-one projections 
onto subspaces $Y_1,\ldots,Y_k$. Then the strong limit $\lim_{n\to\infty} (P_{k}P_{k-1}\cdots P_{1})^nx$ exists 
for each $x\in X$ and defines a norm-one-projection onto the intersection $Y_1\cap\ldots\cap Y_k$.
\end{proposition}

Using this proposition together with the uniqueness of norm-one projections in smooth spaces, 
we obtain the following simple result:

\begin{proposition}
  Let $X$ be a uniformly convex and smooth Banach space and let $P_1$ and $P_2$ be two norm-one 
projections in $X$. Then these projections are consistent, that is, there is a projection $P_{12}$ 
onto the 
intersection of the ranges of $P_1$ and $P_2$ with the property that $P_{12}P_1 = P_{12}P_2 = P_{12}$.
\end{proposition}

\begin{proof}
  By Proposition~\ref{prop:Reich}, the limits
  \[
    P_{12}x = \lim_{n\to\infty} (P_1P_2)^{n}x \qquad \text{and}\qquad P_{12}'x = \lim_{n\to\infty} (P_2P_1)^{n}x
  \]
  both define a norm-one projection onto $\im P_1\cap \im P_2$. These projections satisfy
  \[
    P_{12}'P_1x = \lim_{n\to\infty} (P_2P_1)^{n}P_1x = \lim_{n\to\infty} (P_2P_1)^{n-1}P_2P_1^2x = \lim_{n\to\infty} (P_2P_1)^{n}x = P_{12}'x
  \]
  and
  \[
    P_{12}P_2x = \lim_{n\to\infty} (P_1P_2)^{n}P_2x = \lim_{n\to\infty} (P_1P_2)^{n-1}P_1P_2^2x 
    = \lim_{n\to\infty} (P_1P_2)^{n}x = P_{12}x
  \]
  for all $x \in X$. Since in smooth Banach spaces norm-one projections are unique (see, for example, Theorem~6 in~\cite[p.~356]{CS1970:Projecting}), we necessarily have $P_{12}=P_{12}'$ 
  and this projection has the required properties.
\end{proof}

Note that we cannot drop the assumption that $X$ is smooth. This can be seen in the following four-dimensional example.

\begin{example}
  We consider the space $\mathbb{R}^{4}$ equipped with the norm
  \[
    \|(x,y,z,w)\|=\sqrt{x^2+y^2+z^2+w^2}+|x|+|y|+|z|+|w|
  \]
  which turns it into a uniformly convex but non-smooth space. Take the projections
  \[
    P_1(x,y,z,w)=\left(x,y-\frac{z}{4},0,0\right)\qquad\text{and}\qquad P_2(x,y,z,w) = \left(0,y,0,w\right)
  \]
  which, by
  \begin{align*}
    \|P_1(x,y,z,w)\| & = \sqrt{x^2+\left(y-\frac{z}{4}\right)^2} + |x| + \left|y-\frac{z}{4}\right|\\
                     & \leq \sqrt{x^2+y^2}+\frac{\sqrt{z^2}}{4} +|x| +|y|+\frac{|z|}{4} \\
                     & \leq \sqrt{x^2+y^2+z^2+w^2} + |x|+|y|+\frac{|z|}{2}+|w|
  \end{align*}
  are both norm-one projections. Moreover,
  \[
    (P_1P_2)(x,y,z,w) = \left(0,y,0,0\right)\qquad\text{and}\qquad (P_2P_1)(x,y,z,w) = \left(0,y-\frac{z}{4},0,0\right)
  \]
  and since both elements are already in the intersection of the ranges, the sequence of iterates is constant in both cases. Hence we get
  \[
    \lim_{n\to\infty}(P_1P_2)^{n}\neq \lim_{n\to\infty}(P_2P_1)^{n}.
  \]
  Moreover, note that neither one of these projections has the properties required in the definition of the Oppenheim angle.  
\end{example}

\text{}

\noindent\textbf{Acknowledgments.} The authors would like to thank an anonymous referee for carefully reading the manuscript and for the valuable suggestions.
The third author was partially supported by
the Israel Science Foundation (Grants No. 389/12 and 820/17), the
Fund for the Promotion of Research at the Technion and by the
Technion General Research Fund.



\begin{thebibliography}{99}
%
\bibitem{AmemiyaAndo}
I. Amemiya and T. Andô. \emph{Convergence of random products of contractions
in {H}ilbert space}. \emph{Acta Sci. Math. (Szeged)} 26 (1965), pp. 239--244.
%
\bibitem{BL2010:Geometric}
C. Badea and Y. I. Lyubich. \emph{Geometric, spectral and asymptotic
properties of averaged products of projections in {B}anach spaces}.
\emph{Studia Math.} 201.1 (2010), pp. 21--35.
%
\bibitem{BS2017:Quantified}
C. Badea and D. Seifert. \emph{Quantified asymptotic behaviour of {B}anach
space operators and applications to iterative projection methods}. \emph{Pure
Appl. Funct. Anal.} 2.4 (2017), pp. 585--598.
%
\bibitem{Geometry1}
M. Berger. \emph{Geometry {I}}. Universitext. Translated from the 1977 French
original by M. Cole and S. Levy, Springer-Verlag, Berlin, 2009
%
\bibitem{BS} A. B\"ottcher and I. M. Spitkovsky. \emph{A gentle guide to the basics of two projections theory}. \emph{Linear Algebra Appl.} 432 (2010), pp. 1412--1459

\bibitem{Reich}
R. E. Bruck and S. Reich. \emph{Nonexpansive projections and resolvents of
accretive operators in {B}anach spaces}. \emph{Houston J. Math.} 3.4 (1977),
pp. 459--470.
%
\bibitem{CS1970:Projecting}
H. B. Cohen and F. E. Sullivan. \emph{Projecting onto cycles in smooth,
reflexive {B}anach spaces}. \emph{Pacific J. Math.} 34 (1970), pp. 355--364.
%
\bibitem{DJRSSS2017:Nonoptimality}
O. Darwin, A. Jha, S. Roy, D. Seifert, R. Steele, and L. Stigant \emph{Non-optimality of the greedy algorithm for subspace
orderings in the method of alternating projections}. \emph{Results Math.} 72.1-2
(2017), pp. 979--990.%
\bibitem{Halperin}
I. Halperin. \emph{The product of projection operators}. \emph{Acta Sci. Math.
(Szeged)} 23 (1962), pp. 96--99.
%
\bibitem{Kayalar1988}
S. Kayalar and H. L. Weinert. \emph{Error bounds for the method of alternating
projections}. \emph{Math. Control Signals Systems} 1.1 (1988), pp. 43--59.
%
\bibitem{AnglesInfinite}
A. Knyazev, A. Jujunashvili and M. Argentati. \emph{Angles between infinite
dimensional subspaces with applications to the {R}ayleigh-{R}itz and alternating
projectors methods}. \emph{J. Funct. Anal.} 259.6 (2010), pp. 1323--1345.
%
\bibitem{K2019}
  E. Kopeck\'{a}. \emph{When products of projections diverge}. \emph{J. London Math. Soc.} doi:10.1112/jlms.12322. 
%
\bibitem{KM2014}
  E. Kopeck\'{a} and V. M\"{u}ller. \emph{A product of three projections}. \emph{Studia Math.} 223 (2014), pp. 175--186. 
\bibitem{KP2017:StrangeProducts}
E. Kopecká and A. Paszkiewicz. \emph{Strange products of projections}.
\emph{Israel J. Math.} 219.1 (2017), pp. 271--286.
%
\bibitem{KR2004}
E. Kopecká and S. Reich. \emph{A note on the von {N}eumann alternating
projections algorithm}. \emph{J. Nonlinear Convex Anal.} 5.3 (2004),
pp. 379--386.
%
\bibitem{KV2001:Polyhedral}
P. Krej\v{c}\v{\i} and A. Vladimirov. \emph{Lipschitz continuity of polyhedral
{S}korokhod maps}. \emph{Z. Anal. Anwendungen} 20.4 (2001), pp. 817--844.
%
\bibitem{matrix2000meyer}
C. Meyer. \emph{Matrix analysis and applied linear algebra}. Society
for Industrial and Applied Mathematics (SIAM), Philadelphia, PA, 2000
%
\bibitem{Neumann1949}J. von Neumann. \emph{On rings of operators. {R}eduction theory}.
\emph{Ann. of Math. (2)} 50 (1949), pp. 401--485.
%
\bibitem{Oppenheim2017a}
I. Oppenheim. \emph{Vanishing of cohomology with coefficients in
representations on {B}anach spaces of groups acting on buildings}.
\emph{Comment. Math. Helv.} 92.2 (2017), pp. 389--428.
%
\bibitem{Oppenheim2017}
I. Oppenheim. \emph{Angle criteria for uniform convergence of averaged
projections and cyclic or random products of projections}. \emph{Israel J.
Math.} 223.1 (2018), pp. 343--362.
%
\bibitem{Stiles}
W. Stiles. \emph{Closest-point maps and their products}. \emph{Nieuw Arch.
Wisk. (3)} 13 (1965), pp. 19--29.
%
\bibitem{HYPERPLANES}
A. Vladimirov, L. Elsner and W.-J. Beyn. \emph{Stability and paracontractivity
of discrete linear inclusions}. \emph{Linear Algebra Appl.} 312.1-3 (2000),
pp. 125--134.
\end{thebibliography}
\end{document}